\documentclass[12pt, leqno,twoside]{article}
\usepackage{amssymb}
\usepackage{amsmath}
\usepackage{amsthm}
\usepackage{graphicx}
\usepackage[pagebackref,colorlinks,citecolor=blue,linkcolor=blue]{hyperref}
\usepackage{cite}

\usepackage{amsmath}
\usepackage{amsfonts}
\usepackage{amssymb}
\usepackage{amsmath,bbm,amssymb,amsxtra}
\usepackage{mathrsfs}
\usepackage{enumerate}
\usepackage{caption}
\allowdisplaybreaks
\usepackage{epsfig}
\usepackage{ae}
\usepackage{epstopdf}
\def\vint{\mathop{\mathchoice%
         {\setbox0\hbox{$\displaystyle\intop$}\kern 0.22\wd0%
          \vcenter{\hrule width 0.6\wd0}\kern -0.82\wd0}%
         {\setbox0\hbox{$\textstyle\intop$}\kern 0.2\wd0%
          \vcenter{\hrule width 0.6\wd0}\kern -0.8\wd0}%
         {\setbox0\hbox{$\scriptstyle\intop$}\kern 0.2\wd0%
          \vcenter{\hrule width 0.6\wd0}\kern -0.8\wd0}%
         {\setbox0\hbox{$\scriptscriptstyle\intop$}\kern 0.2\wd0%
          \vcenter{\hrule width 0.6\wd0}\kern -0.8\wd0}}%
         \mathopen{}\int}

\newcommand{\Sobolevp}{ W^{1, p}(\real^{d+1}_+, \mu_{\alpha}^\lambda)}
\newcommand{\Sobolevpp}{ W^{1, p}(\real^{d+1}_+, \mu_{p-1}^\lambda)}
\newcommand{\Trace}{\mathscr {T}}
\newcommand{\tdQ}{ {\mathcal S(Q)}}
\newcommand{\tdP}{ {\mathcal S(P)}}

\newcommand{\mual}{{\mu_\alpha^\lambda}}
\newcommand{\supp}{{\rm supp}\,}

\newcommand{\omegaal}{{\omega_\alpha^\lambda}}

\newcommand{\Lip}{{\rm Lip}\,}

\newcommand{\N}{{\mathbb N}}

\newcommand{\dyadic}{\mathscr{Q}}

\newcommand{\real}{{\mathbb R}}
\newcommand{\dist}{\text{\rm dist}}

\newtheorem{thm}{Theorem}[section]

\newtheorem{cor}[thm]{Corollary}
\newtheorem{lem}{Lemma}[section]
\newtheorem{prop}[thm]{Proposition}

\newtheorem{claim}{Claim}[section]
\newtheorem{subclaim}{Subclaim}
\newtheorem{conj}[equation]{Conjecture}
\newtheorem{case}{Case}[section]
\newtheorem*{mysolution}{Solution}
\newtheorem{step}{Step}[section]
\theoremstyle{definition}
\newtheorem{defn}[thm]{Definition}
\newtheorem{example}[thm]{Example}
\newtheorem{prob}[equation]{Problem}
\newtheorem{ques}[equation]{Question}
\newtheorem{rem}{Remark}[section]
\newtheorem{rems}{Remarks}[section]
\newcounter {own}
\def\theown {\thesection       .\arabic{own}}

\newenvironment{pf}[1][]{%
 \vskip 3mm
 \noindent
 \ifthenelse{\equal{#1}{}}%
  {{\slshape Proof. }}%
  {{\slshape #1.} }%
 }%
{\qed\bigskip}

\newcounter{alphabet}

\makeatletter

\newcommand{\W}{{\mathscr W}}

\newcommand{\id}{{\operatorname{id}}}
\newcommand{\essinf}{{\operatorname{ess\ inf}}}

\newcommand{\RNum}[1]{\uppercase\expandafter{\romannumeral #1\relax}}
\def\be{\begin{equation}}
\def\ee{\end{equation}}

\newcommand{\ben}{\begin{enumerate}}
\newcommand{\een}{\end{enumerate}}

\newcommand{\blem}{\begin{lem}}
\newcommand{\elem}{\end{lem}}
\newcommand{\bthm}{\begin{thm}}
\newcommand{\ethm}{\end{thm}}
\newcommand{\bcor}{\begin{cor}}
\newcommand{\ecor}{\end{cor}}
\newcommand{\beg}{\begin{examp}}
\newcommand{\eeg}{\end{examp}}
\newcommand{\begs}{\begin{examples}}
\newcommand{\eegs}{\end{examples}}
\newcommand{\bdefe}{\begin{defn}}
\newcommand{\edefe}{\end{defn}}
\newcommand{\bprob}{\begin{prob}}
\newcommand{\eprob}{\end{prob}}
\newcommand{\bques}{\begin{ques}}
\newcommand{\eques}{\end{ques}}
\newcommand{\bei}{\begin{itemize}}
\newcommand{\eei}{\end{itemize}}
\newcommand{\bcl}{\begin{claim}}
\newcommand{\ecl}{\end{claim}}
\newcommand{\bscl}{\begin{subclaim}}
\newcommand{\escl}{\end{subclaim}}
\newcommand{\bca}{\begin{case}}
\newcommand{\eca}{\end{case}}
\newcommand{\bstep}{\begin{step}}
\newcommand{\estep}{\end{step}}
\newcommand{\bsol}{\begin{mysolution}}
\newcommand{\esol}{\end{mysolution}}
\newcommand{\bcon}{\begin{conj}}
\newcommand{\econ}{\end{conj}}
\newcommand{\bcons}{\begin{conjs}}
\newcommand{\econs}{\end{conjs}}
\newcommand{\bprop}{\begin{prop}}
\newcommand{\eprop}{\end{prop}}
\newcommand{\br}{\begin{rem}}
\newcommand{\er}{\end{rem}}
\newcommand{\brs}{\begin{rems}}
\newcommand{\ers}{\end{rems}}
\newcommand{\bo}{\begin{obser}}
\newcommand{\eo}{\end{obser}}
\newcommand{\bos}{\begin{obsers}}
\newcommand{\eos}{\end{obsers}}

\newcommand{\bpf}{\begin{pf}}
\newcommand{\epf}{\end{pf}}

\newcommand{\ba}{\begin{array}}
\newcommand{\ea}{\end{array}}
\newcommand{\beq}{\begin{eqnarray}}
\newcommand{\beqq}{\begin{eqnarray*}}
\newcommand{\eeq}{\end{eqnarray}}
\newcommand{\eeqq}{\end{eqnarray*}}

\begin{document}
\title{\Large\bf  Borderline case of  traces and extensions for weighted Sobolev spaces
 \footnotetext{\hspace{-0.35cm}
$2010$ {\bf Mathematics Subject classfication}: 46E35, 42B35, 30L99
\endgraf{{\bf Key words and phases}: Sobolev space, borderline case, trace theorem, Besov-type space, Muckenhoupt  $A_p$ weights}
%\endgraf{Authors have been supported by the Academy of Finland (project No. 323960)}
\endgraf{{\it ${}^{\mathbf{*}}$ Corresponding author}}
}
}
\author{Manzi Huang, Xiantao Wang, Zhuang Wang${}^{\mathbf{*}}$, and Zhihao Xu}
\date{ }
\maketitle
%\tableofcontents
\begin{abstract}
In this paper, we study the traces and the extensions for weighted Sobolev spaces on upper half spaces  when the weights reach to the borderline cases. We first give a full characterization of the existence of trace spaces for these weighted Sobolev spaces, and then  study the trace parts and the extension parts between the weighted Sobolev spaces and  a new kind of Besov-type spaces (on hyperplanes) which are defined by using integral averages over selected layers of dyadic cubes. %It turns out that, when $p=1$, trace spaces for the weighted Sobolev spaces are fully characterized by  our Besov-type spaces  and, when $p>1$, the trace spaces are controlled from both above and below by our  Besov-type spaces with suitable parameters.

\end{abstract}

\section{Introduction}
The study of trace spaces (on the boundary of a domain) for Sobolev spaces on Euclidean domains is associated with the Dirichlet boundary value problem related to elliptic differential equations.  Certain types  of Dirichlet problems are guaranteed to have solutions  when the boundary data  arises as the trace of a Sobolev function. The work of Gagliardo identified the classical Besov spaces $B^{1-1/p}_{p, p}(\real^d)$ as the trace spaces for the first order Sobolev spaces $W^{1, p}(\real^{d+1}_+)$, where $1< p<\infty$, $\real^{d+1}_+:=\{(x, t): x\in \real^d, t>0\}$, and $1\leq d\in \N$, the set of all natural numbers, see {\cite[Theorem 1.\RNum1]{Ga}.} In the borderline case, that is, $p=1$, it was also shown that the trace spaces for $W^{1,1}(\real^{d+1}_+)$ are $L^1(\real^d)$, see  {\cite[Theorem 1.\RNum2]{Ga}.} Here we say that a function space $\mathbb X(\real^d)$ is the trace space for the function space $\mathbb Y(\real^{d+1}_+)$ if every function in $\mathbb Y(\real^{d+1}_+)$ has a trace in $\mathbb X(\real^d)$ and every function in $\mathbb X(\real^d)$ is the trace of some function in $\mathbb Y(\real^{d+1}_+)$. We refer to \cite{Ga} and monographs \cite{Pe,T,T4} for more trace results and properties of classical Besov spaces.

It is natural to seek for the trace  spaces  for Sobolev spaces associated with weights. The trace spaces for Sobolev spaces with the Muckenhoupt $A_p$ weights (in the following, we briefly denote by $A_p$ the set of all Muckenhoupt  $A_p$ weights) (see Definition \ref{A_p-def} below for the precise definition) were also well studied. For example, early trace results for Sobolev spaces with weights in $A_p$ were given by Nikolskii, Lizorkin and Va\v{s}arin, see\cite{Li,Ni,Va}. Some recent  results about trace spaces of  Sobolev spaces with weights belonging to $A_p$ were given by  \cite{KSW17,LS-21,Tyu1,Tyu2,Tyu4}. We also refer  to \cite{Ar, MirRus, Pe, SlBa, T, T4, IV, HaMa}  for more investigations in this line. {%, but this is not an exhaustive list of papers on these topics in the current literatures.}
 There is an advantage for the discussions on the trace spaces of Sobolev spaces when the weights are in $A_p$ because under this constraint, the related trace spaces always exist. However, to the best of our knowledge, there are very few results about the trace spaces for Sobolev spaces with weights not belonging to $A_p$ {  in the literatures.}

For $1\leq p<\infty$, $-1<\alpha\leq p-1$ and $\lambda\in \real$, let
$\omega^{\lambda}_{\alpha}$: $\mathbb{R}^{d+1}_{+}\rightarrow(0,\infty)$ denote the weight
\begin{equation}\label{eq-1.1}
(x_{1},x_{2},...,x_{d+1})\mapsto
\left\{\begin{array}{cl}
|x_{d+1}|^{\alpha}\log^{\lambda}\frac{4}{|x_{d+1}|},&x_{d+1}\in(0,1], \\
\log^{\lambda}4,&\,x_{d+1}\in(1,\infty),
\end{array}\right.
\end{equation}
and the weighted measure $\mu_{\alpha}^\lambda$ on $\real^{d+1}_+$ is defined by
\begin{equation}\label{weight}
\mu^{\lambda}_{\alpha}(E)=\frac{1}{\log^{\lambda}4}\int_{E}\omega_{\alpha}^{\lambda}\, dm_{d+1}.
\end{equation}

It is known that, when $-1<\alpha< p-1$, for each $p\geq 1$, the weight $\omega^{\lambda}_{\alpha}$ belongs to $A_p$ for every $\lambda\in \real$ (cf. Proposition \ref{A_p} below). But when $\alpha=p-1$, the situation is much different. Again, by Proposition \ref{A_p}, we see that if $p>1$, the weight $\omega^{\lambda}_{p-1}$ does  not belong to $A_p$ for  any $\lambda\in \real$, and even for the special case when $p=1$, the weight $\omega^{\lambda}_{0}$ does  not belong to $A_p$ for any $\lambda<0$.

%%%%%%%%%%%%%%%%%%%%%%%%%%%%%%%%%%%%%%%%%%%
\begin{figure}[htbp]
\begin{center}
\includegraphics{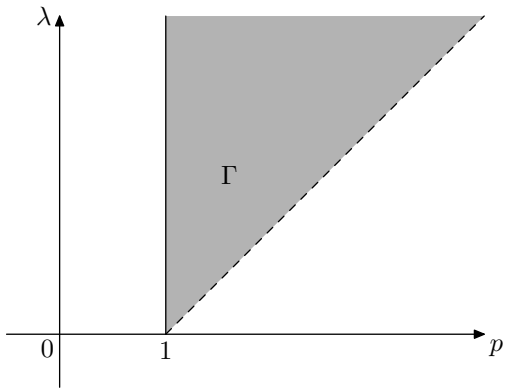}
\end{center}
\caption{The set $\Gamma$.} \label{fig-1}
\end{figure}
%%%%%%%%%%%%%%%%%%%%%%%%%%%%%%%%%%%%%%%%%%%%%%

Let
$$\Pi=\{(p,\lambda)\in \mathbb{R}^2:\; p\geq 1\}$$ and
$$\Gamma=\{(p,\lambda)\in \Pi :\; p>1, \lambda>p-1\}\cup \{(p,\lambda)\in \Pi :\; p=1, \lambda\geq 0\}$$ (see Figure \ref{fig-1}).

The following example shows that there are $(p,\lambda)\in \Pi$ such that the related trace operators $\mathscr T$ on $W^{1, p}(\real^{d+1}_+, \mu_{p-1}^\lambda)$ do not exist. The reader is referred to Definition \ref{trace-defn} below for the precise definition of trace operators.

\begin{example}\label{example}
Suppose that $(p,\lambda)\in \Pi$ and  $\alpha=p-1$. For any pair $(p, \lambda)\notin\Gamma$,  there is a function $u\in \Sobolevp$ such that $\mathscr T u$ does not exist.
\end{example}
Naturally, one will ask whether for any $(p, \lambda)\in\Gamma$, the related trace operator $\mathscr T$ on $W^{1, p}(\real^{d+1}_+, \mu_{p-1}^\lambda)$ exists. Our following result shows that the answer to this question is positive.

\begin{thm}\label{Thm-1.2}
Suppose that $(p,\lambda)\in \Pi$. Then
\ben
\item[$(i)$]
the trace function $\mathscr T u$ belongs to $L^p(\real^d)$ for every $u\in W^{1, p}(\real^{d+1}_+, \mu_{p-1}^\lambda)$ if and only if $(p, \lambda)\in \Gamma$.
\item[$(ii)$]
for every $(p, \lambda)\in \Gamma$, the trace operator $\mathscr T:$ $W^{1, p}(\real^{d+1}_+, \mu_{p-1}^\lambda)\rightarrow L^p(\real^d)$ is bounded and linear.
\een
\end{thm}

To investigate the characterization of the trace spaces for weighted Sobolev spaces $W^{1, p}(\real^{d+1}_+, \mu_{p-1}^\lambda)$ with $(p, \lambda)\in\Gamma$ is our main purpose of this paper.
In \cite{KSW17},  Koskela, Soto and the third author of this paper considered the weights $\omega^{\lambda}_{\alpha}$ and the related measures $\mu_{\alpha}^\lambda$ when $\lambda=0$ and $\alpha<p-1$, and they obtained the characterization of the corresponding trace spaces by the aid of a class of Besov spaces. But they did not consider the borderline case $\alpha=p-1$ because  
 the discussions in \cite{KSW17} do require the constraint $\alpha<p-1$, and thus, the methods are invalid for the case $\alpha=p-1$. 
  Hence it needs new ideas  to deal with this case.

Under the inspiration of the approach of using integral averages over dyadic cubes in \cite{KSW17} and the idea of choosing a system of tilings in \cite{Tyu2}, we construct a new kind of Besov-type spaces $\mathcal B^{\gamma}_{p}(\real^d)$ based on the integral averages over the so-called {\it selected layers of dyadic cubes} (see Definition \ref{Besov-space-0} below for the details). By replacing $L^p(\real^d)$ in Theorem \ref{Thm-1.2}$(ii)$
with $\mathcal B^{\gamma}_{p}(\real^d)$, we get the following related trace result.

\begin{thm}\label{thm-1.3}
Suppose that $(p,\lambda)\in \Gamma$. Then the trace operator $\mathscr T: W^{1, p}(\real^{d+1}_+, \mu_{p-1}^\lambda)\rightarrow \mathcal B^{\gamma}_{p}(\real^d)$
 is bounded and linear provided that $0<\gamma<\lambda-(p-1)$ if $p>1$ or $0<\gamma\leq \lambda$ if $p=1$.
\end{thm}

By the aid of $\mathcal B^{\lambda}_{p}(\real^d)$, we also obtain the following extension result.

\begin{thm}\label{thm-1.4}
 Suppose that $(p,\lambda)\in \Gamma$.
  Then there exists a bounded and linear extension operator $\mathscr E: \mathcal B^{\lambda}_{p}(\real^d) \rightarrow W^{1, p}(\real^{d+1}_+, \mu_{p-1}^\lambda)$ such that $\mathscr T\circ \mathscr E=\id$ on $\mathcal B^{\lambda}_{p}(\real^d)$, where ``$\,\id$'' denotes ``the identity map''.
\end{thm}

%From the definition of the Besov-type spaces $\mathcal B^{\gamma}_{p}(\real^d)$, there is a natural embedding
%\begin{equation}\label{embedding}
%\mathcal B^{\gamma_1}_{p}(\real^d)\subset\mathcal B^{\gamma_2}_{p}(\real^d)\ \ \text{if}\ \ 0<\gamma_2\leq \gamma_1<\infty.
%\end{equation}

%As one might see, when $p>1$, the number $\gamma$ in Theorem \ref{thm-1.3} is not equal to $\lambda$, but in Theorem \ref{thm-1.4}, $\gamma=\lambda$. Hence the combination of Theorem \ref{thm-1.3} and Theorem \ref{thm-1.4} does %not give the full characterization of the trace spaces for weighted Sobolev spaces $W^{1, p}(\real^{d+1}_+, \mu^\lambda_{p-1})$ with $p>1$, but it gives the best upper and lower bounds that we know so far. If we denote by %${\mathrm Tr}(W^{1, p}(\real^{d+1}_+, \mu^\lambda_{p-1}))$   the trace spaces of $W^{1, p}(\real^{d+1}_+, \mu^\lambda_{p-1})$, then it follows that {\red for all $p>1$, $\lambda>p-1$ and arbitrary small $\epsilon>0$,
%$$\mathcal B^{\lambda}_{p}(\real^d)\subset{\mathrm Tr}(W^{1, p}(\real^{d+1}_+, \mu^\lambda_{p-1}))\subset \mathcal B^{\lambda-(p-1)-\epsilon}_{p}(\real^d).$$}
%This implies that the full characterization of the trace spaces for the weighted Sobolev spaces $W^{1, p}(\real^{d+1}_+, \mu^\lambda_{p-1})$ with $p>1$ still remains open.
%But,

The combination of Theorem \ref{thm-1.3} and Theorem \ref{thm-1.4} implies that, when $p=1$, we get the full characterization of the trace spaces for weighted Sobolev spaces $W^{1,1}(\mathbb{R}^{d+1}_{+},\mu_{0}^{\lambda})$, which is formulated in the following corollary.
\begin{cor}\label{cor-1.5}
Let $\lambda>0$. Then  the Besov-type space $\mathcal B^{\lambda}_{1}(\mathbb{R}^{d})$ is the trace space of the weighted Sobolev space $W^{1,1}(\mathbb{R}^{d+1}_{+},\mu_{0}^{\lambda})$.
\end{cor}

 Notice that  if  $\lambda>0$, the weights $\omega_0^{\lambda}$ are actually in $A_1$ (cf. Proposition \ref{A_p}). Hence the trace spaces $\mathcal B^{\lambda}_{1}(\mathbb{R}^{d})$ are equivalent to the ones defined in  \cite{Tyu2}. The  trace spaces in \cite{Tyu2} are complicated since it needs to work for all weights in $A_1$, and then, our concrete trace spaces $\mathcal B^{\lambda}_{1}(\mathbb{R}^{d})$, which can be regarded as an example, would be helpful to understand it.

We remark that we do not consider the case when $\alpha>p-1$. This is because, when $\alpha>p-1$, by slightly modifying the proof of Proposition \ref{A_p} and the construction of Example \ref{example}, we shall know that  for any $(p, \lambda)$ with $p\geq 1$ and $\lambda\in \real$, the weight $\omega_\alpha^\lambda$ does not belong to $A_p$, and the trace operator $\mathscr T$ on $W^{1, p}(\real^{d+1}_+, \mu_\alpha^\lambda)$ does not exist either.

%Our work is fully under the Euclidean setting and we would like to mention that the trace theory for function spaces on metric measure spaces are under developing quiet recently. We refer the %interested readers to  \cite{GKS, LS, Ma, MSS, SS, LLW, KW20, W21, KNW} for trace results in the setting of metric measure spaces.

The paper is organized as follows. In Section \ref{sec-2}, some necessary terminology will be introduced and a proposition will be proved, which classifies the weights $w^\lambda_\alpha$ into two classes: $A_p$ class and non $A_p$ class. The proofs of Example \ref{example},  Theorem \ref{Thm-1.2} and Theorem \ref{thm-1.3} will be presented in Section \ref{sec-3}. Theorem \ref{thm-1.4} will be shown in Section \ref{sec-4}.

\section{Preliminaries}\label{sec-2}
Throughout this paper, the letter $C$ (sometimes with a subscript) will denote a positive constant that usually depend only on the given parameters of the spaces and may change at different occurrences; if $C$ depends on $a,$ $b,$ $\ldots$, then we write $C = C(a,b,\ldots).$ The notation $A\lesssim B$ (resp. $A \gtrsim B$) means that there is a constant $C\geq 1$ such that $A \leq C \cdot B$ (resp. $A \geq C \cdot B).$  If $A\lesssim B$ and $A \gtrsim B$, then we write $A\approx B$. If $(X, \mu)$ is a measure space, for every function $f\in L^1_{\rm loc}(X, \mu)$ and every measurable subset $A\subset X$, let $\vint_Afd\mu$ stand for the integral average $\frac{1}{\mu(A)}\int_Afd\mu$, i.e.,
$$\vint_Afd\mu=\frac{1}{\mu(A)}\int_Afd\mu.$$

Let us recall the definition of the Muckenhoupt $A_p$ weights with $1\leq p<\infty$ (see, e.g. \cite[Charpter 15]{HKM} or \cite{DILTV}).

\begin{defn}\label{A_p-def}
We say that a weight $\omega: \real^{d+1}_+\rightarrow [0, \infty)$ belongs to the Muckenhoupt class $A_p$ if $\omega$ is locally integrable and there is a constant $C>0$ such that for any cube $Q\subset \real^{d+1}_+$,
\[\left(\vint_Q \omega\, dm_{d+1}\right)\left(\vint_Q \omega^{-1/(p-1)}\, dm_{d+1}\right)^{p-1}\leq C\ \ \ \text{if}\ \ p>1\]
and
\[\left(\vint_Q \omega\, dm_{d+1}\right)\leq C \,\essinf_{x\in Q} \, {\omega(x)} \ \ \ \text{if}\ \ p=1.\]
\end{defn}

\begin{prop}\label{A_p}
\noindent $(a)$ Suppose that $(p, \lambda)\in \Pi$, and let $-1<\alpha<p-1$. Then the weights $\omega_\alpha^\lambda$ $($see \eqref{eq-1.1} for the definition$)$ belong to $A_p$ for all pairs $(p, \lambda)$ with $p\geq 1$ and $\lambda\in \real$.

\noindent $(b)$ Suppose that   $(p, \lambda)\in \Pi$, and let $\alpha=p-1$. Then the weights $\omega_\alpha^\lambda$  belong to $A_p$ if and only if $p=1$ and $\lambda\geq 0$.
 %\ben
%\item[$(i)$]
%If $p>1$, then the weight $\omega_\alpha^\lambda$ does not belong to $A_p$ for any $\lambda\in \real$.
%\item[$(ii)$]
%If $p=1$, then the weight $\omega_\alpha^\lambda$  belongs to $A_1$ if and only if $\lambda\geq 0$.
%\een
\end{prop}
\begin{proof}
$(a)$ The proof follows from a direct computation. We omit it here.

$(b)$ We first show that if $p=1$ and $\lambda\geq 0$, the weights $\omega_\alpha^\lambda$  belong to $A_p$. If $p=1$ and $\lambda=0$, then $\omega_0^\lambda=1$, and obviously it belongs $A_1$. If  $p=1$ and $\lambda>0$, then for any $t\geq0$ and $h>0$, by letting $Q=(0,h]^{d}\times(t,t+h]$ and by using the integration by parts, we have
\begin{equation}\label{tmp-A_1}
\vint_Q \omega_0^\lambda\, dm_{d+1}\leq C_{\lambda}\log^{\lambda}\frac{4}{t+h}=C_{\lambda}\cdot\essinf_{x\in Q} \omega_0^\lambda(x),
\end{equation}
which implies that the weight $\omega_0^\lambda$ belongs to $A_1$.

For the converse, it suffices to show that for any pair $(p, \lambda)$ with $p>1$ and $\lambda\in \real$ or $p=1$ and  $\lambda<0$, the weight $\omega_\alpha^\lambda$ does not belong to $A_p$.
 For  $p>1$, we choose the cube $Q=(0,2^{-k}]^{d+1}\subset\real^{d+1}_{+}$. Then
$$
\vint_Q |x_{d+1}|^{\alpha}\log^{\lambda}\frac{4}{|x_{d+1}|}\, dm_{d+1}=2^{k}\int_{0}^{2^{-k}}t^{p-1}\log^{\lambda}\frac{4}{t}dt
$$
and
$$
\vint_Q \left(|x_{d+1}|^{\alpha}\log^{\lambda}\frac{4}{|x_{d+1}|}\right)^{-1/(p-1)}\, dm_{d+1}=2^{k}\int_{(k+2)\log2}^{\infty}t^{-\frac{\lambda}{p-1}}dt.
$$
It follows that if $\lambda>p-1$, then
$$
\left(\vint_Q |x_{d+1}|^{\alpha}\log^{\lambda}\frac{4}{|x_{d+1}|}\, dm_{d+1}\right)\left(\vint_Q (|x_{d+1}|^{\alpha}\log^{\lambda}\frac{4}{|x_{d+1}|})^{-1/(p-1)}\, dm_{d+1}\right)^{p-1}\gtrsim(k+2)^{p-1},
$$
and if $\lambda\leq p-1$, then for each $k\geq0$, $$\int_{(k+2)\log2}^{\infty}t^{-\frac{\lambda}{p-1}}dt=\infty.$$ Hence $\omega_\alpha^\lambda$ does not belong to $A_p$ for {any $\lambda \in \real$}. For $p=1$, the estimate \eqref{tmp-A_1} will fail for any $\lambda<0$, since if $t=0$, we have $\essinf_{x\in Q} \omega_0^\lambda(x)=0$. This implies that $\omega_0^\lambda$ does not belong to $A_1$ for any $\lambda<0$. Hence the proposition is proved.
%
%For the remaining case when $p=1$, if $\lambda=0$, then $\omega_0^\lambda=1$, and obviously it belongs $A_1$. If $\lambda>0$, then for any $t\geq0$ and $h>0$, by letting $Q=(0,h]^{d}\times(t,t+h]$ and by using the integration by parts, we have
%\begin{equation}\label{tmp-A_1}
%\vint_Q \omega_0^\lambda\, dm_{d+1}\leq C_{\lambda}\log^{\lambda}\frac{4}{t+h}=C_{\lambda}\cdot\essinf_{x\in Q} \omega_0^\lambda(x),
%\end{equation}
%which implies that the weight $\omega_0^\lambda$ belongs to $A_1$.
%But if $\lambda<0$, the estimate \eqref{tmp-A_1} will fail since for $t=0$, we have $\essinf_{x\in Q} \omega_0^\lambda(x)=0$. This implies that $\omega_0^\lambda$ does not belong to $A_1$. Hence the proposition is proved.
\end{proof}

Let us recall that $\mu_{\alpha}^\lambda$ denotes the weighted measure on $\real^{d+1}_+$ defined as in \eqref{weight}.
Then by a direct computation, we know that the measure $\mual$ is doubling on $\real^{d+1}_+$, i.e., there exists a constant $C\geq 1$ such that for all $x\in \real^{d+1}_+$  and $r>0$,
\[\mual(\mathbb{B}(x, 2r)\cap \real^{d+1}_+)\leq C\mual(\mathbb{B}(x, r)\cap\real^{d+1}_+ ),\]
where $\mathbb{B}(x, r)$ denotes the open ball with center $x$ and radius $r$.

\begin{defn}
Suppose that $p\in[1,\infty)$. Then $W^{1,p}(\mathbb{R}^{d+1}_+,\mu^{\lambda}_{\alpha})$ is defined as the normed space of all measurable functions $f\in L^1_{\rm loc}(\mathbb{R}^{d+1}_+)$ such that their first-order distributional derivatives, denoted by $\nabla f$, belong to $L^1_{\rm loc}(\mathbb{R}^{d+1}_+)$, and
$$
\|f\|_{W^{1,p}(\mathbb{R}^{d+1}_+,\mu^{\lambda}_{\alpha})}:=\|f\|_{L^{p}(\mathbb{R}^{d+1}_+,\mu^{\lambda}_{\alpha})}+\|\nabla f\|_{L^{p}(\mathbb{R}^{d+1}_+,\mu^{\lambda}_{\alpha})}<+\infty.$$
\end{defn}

In order to formulate the dyadic norms of the related Besov-type spaces, let us recall the standard dyadic decompositions of $\real^d$ and $\real^{d+1}_+$, respectively (cf. \cite[Section 2]{KSW17}).

Denote by $\dyadic_d$ the collection of dyadic semi-open cubes in $\real^d$, i.e., the cubes of the form $Q := 2^{-k}\big((0,1]^d + m\big)$, where $k \in \mathbb Z$, the set of all integers, $m \in {\mathbb Z}^d$, and $\dyadic^+_d$ stands for the set of all cubes in $\dyadic_d$ which are contained in the upper half-space $\real^{d-1}\times(0,\infty)$. Write $\ell(Q)$ for the edge length of $Q \in \dyadic_d$, i.e., $2^{-k}$ in the preceding representation, and $\dyadic_{d,k}$ for the cubes $Q \in \dyadic_d$ such that $\ell(Q) = 2^{-k}$.

 Let $Q\in \dyadic_{d, 2^j}$ for some $j\in \N$. We say that $Q'$ in $\dyadic_d$ is a {\it selected neighbor} of $Q$, denoted by $Q' \asymp Q$, if $Q'\in \dyadic_{d, 2^j}\cup \dyadic_{d, 2^{j-1}}$ and $\overline{Q}\cap\overline{Q'} \neq \emptyset$.  Here we unify $\dyadic_{d, 2^{-1}}$ as $\dyadic_{d, 1}$, i.e., $\dyadic_{d, 2^{-1}}=\dyadic_{d, 1}$. Note that for every $Q\in\dyadic_d$, the number of its neighbors is uniformly bounded, and for every $Q\in \cup_{j\in \N}\dyadic_{d, 2^j}$, also, the number of its selected neighbors is uniformly bounded.

\begin{defn}\label{Besov-space-0}
Let $p\in [1, \infty)$ and $\lambda>0$. Then the Besov space $\mathcal B^{\lambda}_{p}(\mathbb{R}^{d})$ is defined as the normed space of all measurable functions $f\in L_{\rm loc}^{1}(\mathbb{R}^{d})$ such that
$$
\|f\|_{\mathcal B^{\lambda}_{p}(\mathbb{R}^{d})}:=\|f\|_{L^{p}(\mathbb{R}^{d})}+\|f\|_{\dot{\mathcal B}^{\lambda}_{p}(\mathbb{R}^{d})}<+\infty,$$
 where $$\|f\|^p_{\dot{\mathcal B}^{\lambda}_{p}(\mathbb{R}^{d})}:=
\sum_{j=0}^{\infty}(2^{j}+2)^{\lambda}\sum_{Q\in\dyadic_{d,2^{j}}}m_{d}(Q)\sum_{Q'\asymp Q}|f_{Q}-f_{Q'}|^p.$$
\end{defn}

%In the norm of  Besov-type space $\mathcal B^{s, \gamma}_{p, p}(\real^d)$, we consider the differences of integral averages over all layers of dyadic cubes with their neighbors; but for the norm of Besov-type space $\mathcal B^{\gamma}_{p}(\real^d)$, we only consider   the differences of integral averages over selected layers ($\{2^k\}_{k\in \N}$ precisely) of dyadic cubes with their selected neighbors.

Let us recall the standard $(1,1)$-Poincar\'e inequality satisfied by the functions that are locally $W^{1,1}$-regular in the upper half-space. If $Q$ is a cube in $\real^{d+1}_+$ such that $\dist(Q,\real^d\times\{0\}) > 0$ and $f \in W^{1,1}(Q)$, then there is a constant $C>0$ independent of $Q$ and $f$ such that
\begin{equation}\label{poincare}
  \vint_Q |f-f_{Q}|\, dm_{d+1}\leq C \ell(Q)\vint_Q |\nabla f| dm_{d+1},
\end{equation} where ``$\dist$'' means ``distance''.

To formulate the trace and extension operators, we first recall a Whitney decomposition of $\real^{d+1}_+$ related to the dyadic decomposition (cf. \cite{KSW17}).  For $Q \in \dyadic_{d,k}$ and $k \in {\mathbb Z}$, write $\W(Q) := Q \times (2^{-k},2^{-k+1}] \in \dyadic^+_{d+1,k}$. To simplify the notation in the sequel, we further define $\dyadic^1_d := \cup_{k \geq 1} \dyadic_{d,k}$.
Then $\{\W(Q) \,:\, Q \in \dyadic_d\}$ is a {\it Whitney decomposition} of $\real^{d}\times(0,\infty)$ with respect to the boundary $\real^d\times\{0\}$. For every $Q \in \dyadic^1_d$, define a smooth function $$\psi_Q \colon \real^{d+1}_{+}\to[0,1]$$ such that\begin{enumerate}
\item[(i)]
  $\Lip \psi_Q \lesssim 1/\ell(Q)$,
\item[(ii)]
   $\inf_{x \in \W(Q)} \psi_Q(x) > 0$ uniformly in $Q \in \dyadic^1_d$,
\item[(iii)]
    $\supp \psi_Q$ is contained in a $\ell(Q)/4$-neighborhood of $\W(Q)$, and
    \item[(iv)]
$\sum_{Q \in \dyadic^1_d} \psi_Q \equiv 1$ in $\bigcup_{Q \in \dyadic^1_d} \W(Q)$,
\end{enumerate} where ``$\supp$'' means ``support''.

 We say that $Q$ and $Q'$ in $\dyadic_d$ are {\it neighbors} and write $Q \sim Q'$ if $\frac12 \leq \ell(Q)/\ell(Q') \leq 2$ and $\overline{Q}\cap\overline{Q'} \neq \emptyset$.
We remark that the sum $\sum_{Q \in \dyadic^1_d}\psi_Q$ is locally finite -- more precisely, it follows from the definition that
\begin{equation}\label{eq:bump-support}
  \supp \psi_Q \cap \supp \psi_{Q'} \neq \emptyset \quad \text{if and only if} \quad Q \sim Q'.
\end{equation}

\begin{defn}[Trace]\label{trace-defn}
Suppose that $f\in L^1_{\rm loc}(\real^{d+1}_+)$ and  $k\in\mathbb{N}$. Define the function $\mathscr{T}_{k}f:\mathbb{R}^{d}\rightarrow \real$ by
\begin{equation*}
\mathscr{T}_{k}f:=\sum_{Q\in\dyadic_{d,k}}\left(\vint_{\mathscr{N}(Q)}f\, dm_{d+1}\right)\chi_{Q},
\end{equation*}
where $\mathscr{N}(Q)=\frac{5}{4}\mathscr{W}(Q):=\{y\in \mathbb{R}^{d+1}_{+}: \dist(y,\mathscr{W}(Q))<\frac{1}{4}\ell(Q)\}$, and define the trace function $\mathscr T f$ by setting
\begin{equation*}\label{trace}
\mathscr T f=\lim_{k\rightarrow \infty} \mathscr T_k f,
\end{equation*}
if the limit exists $m_d$-a.e.  in $\real^d$.
\end{defn}

Before going to the definition of the extension operator,  we give one more notation. For any $Q\in \dyadic_d$, let $\tdQ$ be the unique cube such that there is $j\in \N$ satisfying $\tdQ\in \dyadic_{d, 2^j}$ and $Q\subset \tdQ$ if $Q\in \bigcup_{k=2^j}^{2^{j+1}-1}\dyadic_{d, k}$.

\begin{defn}[Extension] \label{extension-defn}
Suppose that $f \in  L^1_{\rm loc}(\real^d)$. Then the {\it selected Whitney extension} $\mathscr E f \colon \real^{d+1}_+\to\real$ is defined by
\[
\mathscr E f(x)= \sum_{Q\in \dyadic^1_d}\left(\vint_{\tdQ} f\, dm_d\right) \psi_Q(x).\]
\end{defn}
It is easy to see that the extension operator $\mathscr E: L^1_{\rm loc}(\real^d)\rightarrow C^\infty(\real^{d+1}_+)$ is   linear.

\section{Trace operators and trace spaces}\label{sec-3}
In this section, we shall give the proofs of Example \ref{example}, Theorem \ref{Thm-1.2} and Theorem \ref{thm-1.3}.

%
%\begin{example}\label{example}
%Let  $p\in [1, +\infty)$, $\alpha=p-1$ and $\lambda\in \real$. For any pair $(p, \lambda)\notin\Gamma$,  there is a function $u\in \Sobolevp$ such that $\mathscr T u$ does not exist.
%\end{example}
\subsection*{Proof of Example \ref{example}}
Recall that if $(p, \lambda)\in \Pi$, but $(p, \lambda)\notin \Gamma$, then $\lambda\leq p-1$ if $p>1$, and $\lambda<0$ if $p=1$.
Let $\varphi$ be a compactly supported smooth function on $\real^d$ with $\varphi(x')=1$ for $x'\in [-1, 1]^d$ and $\supp \varphi\subset [-2, 2]^d$. Then we define the function $u$ as follows.

For $(x', t)\in \real^{d+1}_+$, let
\[u(x', t)=\varphi(x')\cdot \max\{ v(t), 0\},\]
where
$$v(t)=\int_{t}^1 \frac{1}{t\log(e/t)\Big(1+\log^\beta\big(\log(e/t)\big)\Big)}\, dt.$$
Here we choose $\beta$ such that $\beta=0$ if $p=1$ and $0<\beta<1<\beta p$ if $p>1$.
Obviously, for any $x'\in [-1, 1]^d$, $\Trace u(x')=\infty$. This shows that $\mathscr T u$ does not exist.

Next, we demonstrate that $u\in \Sobolevp$.
For any $(x', t)\notin [-2, 2]^d\times (0, 1]$, we have that
$$|u(x', t)|=|\nabla u(x', t)|=0.$$
For any $(x', t)\in [-2, 2]^d\times (0, 1]$, there is a constant $C>0$ such that
$$|u(x', t)| \leq C |v(t)| \leq C \int_t^1 \frac{1}{t\log(e/t)}\, dt=C \log\big(\log(e/t)\big)$$
and
\begin{align*}|\nabla u(x', t)|&\leq C|v(t)|+ \frac{C}{t\log(e/t)\log^\beta\big(\log(e/t)\big)}\\
&\leq  C \log\big(\log(e/t)\big)+ \frac{C}{t\log(e/t)\log^\beta\big(\log(e/t)\big)}.
\end{align*}
For $p=1$, $\lambda<0$ and $\beta=0$,
the above facts guarantee that
\[\|u\|_{L^1(\real^{d+1}_+,\mual)} \lesssim 4^d\int_{0}^{1} \log^{\lambda}(4/t)\log\big(\log(e/t)\big)\, dt<\infty\]
and
\begin{align*}
\|\nabla u\|_{L^1(\real^{d+1}_+,\mual)} &\lesssim \|u\|_{L^1(\real^{d+1}_+,\mual)} +4^d\int_0^1 \frac{1}{t\log(e/t)\log^{-\lambda}(4/t)}\, dt\\
&\lesssim \|u\|_{L(\real^{d+1}_+,\mual)} +\int_0^1 \frac{1}{t\log^{1-\lambda}(4/t)}\, dt<\infty.
\end{align*}
For $p>1$, $\lambda\leq p-1$ and $1/p <\beta<1$, similarly, we obtain that
\[\|u\|^p_{L^p(\real^{d+1}_+,\mual)} \leq 4^d\int_{0}^{1} t^{p-1}\log^{\lambda}(4/t)\log^p\big(\log(e/t)\big)\, dt<\infty\]
and
\begin{align*}
\|\nabla u\|^p_{L^p(\real^{d+1}_+,\mual)} &\lesssim \|u\|^p_{L^p(\real^{d+1}_+,\mual)} +4^d\int_0^1 \frac{1}{t\log^p(e/t)\log^{-\lambda}(4/t) \log^{\beta p}\big(\log(e/t)\big)}\, dt\\
&\lesssim \|u\|^p_{L^p(\real^{d+1}_+,\mual)} +\int_0^1 \frac{1}{t\log^{p-\lambda}(4/t) \log^{\beta p}\big(\log(4/t\big))}\, dt<\infty.
\end{align*}
 Hence $u\in \Sobolevp$. These imply that the example is true.
\qed

%\subsection*{Proof of Theorem \ref{thm-1.1}}
%
%The proof of Theorem \ref{thm-1.1} follows by modifying the proof of Theorem 1.1 in \cite{KSW17}. We omit the details here.
%\qed

\subsection*{Proof of Theorem \ref{Thm-1.2}}

{{ \bf Proof of (i)}.} From  Example \ref{example}, it suffices to show that if $(p, \lambda)\in \Gamma$, then the trace function $\mathscr T u$ belongs to $L^p(\real^d)$ for every $u\in W^{1, p}(\real^{d+1}_+, \mu_{p-1}^\lambda)$.

Let $f\in \Sobolevp$ with $\alpha=p-1$ and $(p, \lambda)\in \Gamma$. The first thing is to verify the existence of the limit in Definition \ref{trace-defn}. To reach this goal, it suffices to show that the function
\[f^*:=\sum_{k\geq 0} |\Trace_{k+1} f-\Trace_k f|+|\Trace_0 f|\]
belongs to $L^p(\real^d)$ because $f^*\in L^p(\real^d)$ implies that $f^*(x)<\infty$ for $m_d$-almost every $x\in \real^d$.

For any $P\in \dyadic_{d, 0},$ notice that $m_{d+1}(\mathscr N(P))\approx 1$ and $\omegaal\approx 1$ in $\mathscr N(P)$, it follows from the Minkowski inequality that
\begin{align}
\left(\int_{\real^d}|f^*|^{p}\, dm_d\right)^{1/p}&\lesssim \left(\int_{\real^d} \left(\sum_{k\geq 0}  |\Trace_{k+1} f-\Trace_k f|\right)^p\, dm_d\right)^{1/p}+\left(\sum_{P\in\dyadic_{d, 0}}\int_{\mathscr N(P)} |f|^p\, dm_{d+1}\right)^{1/p}\notag\\
& \lesssim \sum_{k\geq 0}\left(\int_{ \real^d} |\Trace_{k+1} f-\Trace_k f|^p\, dm_d\right)^{1/p} + \left(\sum_{P\in\dyadic_{d, 0}}\int_{\mathscr N(P)} |f|^p\, d\mual\right)^{1/p}\label{integral-P}.
\end{align}

For any $x\in \real^d$, let $Q_k^x$ be the unique cube in $\dyadic_{d, k}$ containing $x$. By the definition, we know that the intersection  $\mathscr N(Q_k^x)\bigcap\mathscr N(Q_{k+1}^x)$ contains a cube $\hat Q$ with edge length comparable to $2^{-k}$. Since $\omegaal(y)\approx 2^{-k\alpha}(2+k)^{\lambda}$ for all $y\in \mathscr N(Q_k^x)$ and $\mual(\mathscr N(Q_k^x))\approx 2^{-k\alpha} (2+k)^{\lambda} m_{d+1}(\mathscr N(Q_k^x))$, it follows from the Poincar\'e inequality \eqref{poincare} that
\begin{align*}
|\Trace_{k+1} f(x)-\Trace_k f(x)|&=\bigg|\vint_{\mathscr N(Q_k^x)} f\, dm_{d+1}-\vint_{\mathscr N(Q_{k+1}^x)} f\, dm_{d+1}\bigg|\\
&\leq \bigg|\vint_{\mathscr N(Q_k^x)} f\, dm_{d+1}-\vint_{\hat Q} f\, dm_{d+1}\bigg|+\bigg|\vint_{\hat Q} f\, dm_{d+1}-\vint_{\mathscr N(Q_{k+1}^x)} f\, dm_{d+1}\bigg|\\
&\lesssim \vint_{\mathscr N(Q_k^x)}|f-f_{\mathscr N(Q_k^x)}|\, dm_{d+1} +\vint_{\mathscr N(Q_{k+1}^x)}|f-f_{\mathscr N(Q_{k+1}^x) }|\, dm_{d+1}\\
&\lesssim 2^{-k}\vint_{\mathscr N(Q_k^x)}|\nabla f|\, dm_{d+1} + 2^{-k}\vint_{\mathscr N(Q_{k+1}^x)}|\nabla f|\, dm_{d+1}\\
&\approx 2^{-k}\vint_{\mathscr N(Q_k^x)}|\nabla f|\, d\mual + 2^{-k}\vint_{\mathscr N(Q_{k+1}^x)}|\nabla f|\, d\mual.
\end{align*}
Applying the H\"{o}lder inequality, we arrive at the estimate
\begin{equation*}\label{estimate-T}
|\Trace_{k+1} f(x)-\Trace_k f(x)|\lesssim 2^{-k} \left(\vint_{\mathscr N(Q_k^x)}|\nabla f|^p\, d\mual\right)^{1/p}+2^{-k} \left(\vint_{\mathscr N(Q_{k+1}^x)}|\nabla f|^p\, d\mual\right)^{1/p}.
\end{equation*}
Hence
\begin{align}
\int_{\real^d}  |\Trace_{k+1} f-\Trace_k f|^p\, dm_d&=\sum_{Q\in \dyadic_{d, k}} \int_Q|\Trace_{k+1} f-\Trace_k f|^p\, dm_d(x)\notag\\
&\lesssim \sum_{Q\in \dyadic_{d, k}} m_d(Q)\bigg(2^{-kp} \vint_{\mathscr N(Q)}|\nabla f|^p\, d\mual \bigg.\notag\\
&\bigg.\quad\quad\quad\quad\quad\quad\quad\quad\quad+\sum_{\substack{Q'\in \dyadic_{d, {k+1}}\\ Q'\subset Q} }2^{-kp} \vint_{\mathscr N(Q')}|\nabla f|^p\, d\mual \bigg)\notag\\
&\lesssim 2^{-k(d+p)} \sum_{Q\in \dyadic_{d, k}\cup\dyadic_{d, k+1}} \vint_{\mathscr N(Q)} |\nabla f|^p\, d\mual\notag\\
&\approx (2+k)^{-\lambda}\sum_{Q\in \dyadic_{d, k}\cup\dyadic_{d, k+1}} \int_{\mathscr N(Q)} |\nabla f|^p\, d\mual,\label{estimate-T_k-T}
\end{align}
since $\alpha=p-1$ implies that $\mual(\mathscr N(Q))\approx 2^{-k(d+p)} (2+k)^{\lambda}$ for all $Q\in \dyadic_{d, k}\cup\dyadic_{d, k+1}$.

Plugging this into \eqref{integral-P}, we obtain that
\begin{align*}
\|f^*\|_{L^p(\real^d)} \lesssim \sum_{k\geq 0} (2+k)^{-\lambda/p}\left(\sum_{Q\in \dyadic_{d, k}\cup\dyadic_{d, k+1}} \int_{\mathscr N(Q)} |\nabla f|^p\, d\mual\right)^{1/p} +\|f\|_{L^p(\real^{d+1}_{+},\mual)}.
\end{align*}

If $p=1$ and $\lambda\geq 0$, it is obvious that
\[\|f^*\|_{L^1(\real^d)} \lesssim\sum_{k\geq 0}\sum_{Q\in \dyadic_{d, k}\cup\dyadic_{d, k+1}} \int_{\mathscr N(Q)} |\nabla f|\, d\mual +\|f\|_{L^1(\real^{d+1}_{+},\mual)} \lesssim \|f\|_{W^{1,1}(\real^{d+1}_+, \mual)}.\]

If $p>1$ and $\lambda>p-1$, since the sum of $(2+k)^{-\lambda/(p-1)}$ converges, it follows from the H\"{o}lder inequality that
\begin{align*}
\|f^*\|_{L^p(\real^d)}& \lesssim \|f\|_{L^p(\real^{d+1}_{+},\mual)}+\left(\sum_{k\geq 0} \sum_{Q\in \dyadic_{d, k}\cup\dyadic_{d, k+1}} \int_{\mathscr N(Q)} |\nabla f|^p\, d\mual\right)^{1/p}\lesssim \|f\|_{\Sobolevp}.
\end{align*}
Thus we obtain that the estimate
\begin{equation}\label{tmp-norm}
\|f^*\|_{L^p(\real^d)}\lesssim \|f\|_{\Sobolevpp}
\end{equation}
 holds for all $1\leq p<\infty$. Hence $f^*<\infty$ for $m_d$-almost every $x\in \real^d$, so the trace function $\Trace f$ exists.

 Since $|\mathscr Tf|\leq |f^*|$ a.e. in $\real^d$, it follows from \eqref{tmp-norm} that
 \begin{equation}\label{add-7-1}
 \|\mathscr Tf\|_{L^p(\real^d)}\lesssim \|f\|_{\Sobolevpp}<\infty.
 \end{equation}
 Hence the trace function $\mathscr T u$ belongs to $L^p(\real^d)$.

{{\bf Proof of (ii)}.}   For every $(p, \lambda)\in \Gamma$, it follows from the   proof of (i) that  the trace function $\mathscr T u$ belongs to $L^p(\real^d)$ for every $u\in W^{1, p}(\real^{d+1}_+, \mu_{p-1}^\lambda)$ and that the norm estimate \eqref{add-7-1} holds for every $f\in W^{1, p}(\real^{d+1}_+, \mu_{p-1}^\lambda)$. These guarantee the boundedness of the trace operator $\mathscr T$. Since the linearity of the trace operator $\Trace$ is obvious from the definition, the proof is complete.
\qed

\subsection*{Proof of Theorem \ref{thm-1.3}}

Let  $f\in W^{1,p}(\mathbb{R}^{d+1}_{+},\mu_{p-1}^{\lambda})$ for $(p, \lambda)\in \Gamma$ (i.e., $\alpha=p-1$). It follows from Theorem \ref{Thm-1.2} that the trace function $\Trace f$ exists and $$\|\mathscr Tf\|_{L^p(\real^d)}\lesssim \|f\|_{\Sobolevp}.$$

The remaining task is to estimate the $\dot {\mathcal B}^{\gamma}_{p}(\mathbb{R}^{d})$-energy of $\mathscr{T}f=:{\tilde f}$.  Let $P\in \dyadic_{d, 2^k}$. Then
$$m_d(P)\sum_{Q\asymp P} |{\tilde f}_Q-{\tilde f}_P|^p \leq \sum_{\substack{Q\asymp P\\ Q\in \dyadic_{d, 2^k}}} \int_P |{\tilde f}-{\tilde f}_Q|^p\, dm_d+\sum_{\substack{Q\asymp P\\ Q\in \dyadic_{d, 2^{k-1}}}} \int_P |{\tilde f}-{\tilde f}_Q|^p\, dm_d=:H^1_P+H^2_P,$$
and so,
\begin{align*}
\|\mathscr Tf\|^{p}_{\dot {\mathcal B}^{\gamma}_{p}(\mathbb{R}^{d})}=\|\tilde f\|^{  p}_{\dot {\mathcal B}^{\gamma}_{p}(\mathbb{R}^{d})}&=\sum_{k=0}^{\infty}(2^{k}+2)^{\gamma}\sum_{P\in\dyadic_{d,2^{k}}}m_{d}(P)\sum_{Q\asymp P}|{\tilde f}_{Q}-{\tilde f}_{P}|^p\\
&\leq \sum_{k=0}^{\infty}(2^{k}+2)^{\gamma}\sum_{P\in\dyadic_{d,2^{k}}}(H^1_P+H^2_P)\\
&=:H^1+H^2,
\end{align*} where
$$H^i=\sum_{k=0}^{\infty}(2^{k}+2)^{\gamma}\sum_{P\in\dyadic_{d,2^{k}}}H^i_P\ \ \ \text{for }\ i=1, 2.$$

Towards the estimate of $H^1$, notice that
\begin{equation*}\label{measure-relation}
m_d(P)\approx m_d(Q) \approx 2^{2^kp}(2^k+2)^{-\lambda}\mu_\alpha^\lambda(\mathscr N(P))\approx 2^{2^kp}(2^k+2)^{-\lambda}\mu_\alpha^\lambda(\mathscr N(Q))
\end{equation*}
for any $P, Q\in \dyadic_{d, 2^k}$ with $P\asymp Q$.
Hence
\begin{align*}
H^1_P&\leq \sum_{\substack{Q\asymp P\\ Q\in \dyadic_{d, 2^k}}}\left( \int_P|{\tilde f}-f_{\mathscr{N}(P)}|^p\, dm_d+m_d(P)|f_{\mathscr{N}(P)}- f_{\mathscr{N}(Q)}|^p+\int_Q|{\tilde f}-f_{\mathscr{N}(Q)}|^p\, dm_d\right)\\
&\lesssim \sum_{\substack{Q\asymp P\\ Q\in \dyadic_{d, 2^k}}} \int_{Q} |{\tilde f}(x)-\mathscr T_{2^k} f(x)|^p\, dm_d(x)+\sum_{\substack{Q\asymp P\\ Q\in \dyadic_{d, 2^k}}}m_d(P)|f_{\mathscr{N}(P)}-f_{\mathscr{N}(Q)}|^p\\
&\approx \sum_{\substack{Q\asymp P\\ Q\in \dyadic_{d, 2^k}}} \int_{Q} |{\tilde f}(x)-\mathscr T_{2^k} f(x)|^p\, dm_d(x)+ 2^{2^kp}(2^k+2)^{-\lambda}\sum_{\substack{Q\asymp P\\ Q\in \dyadic_{d, 2^k}}}\mu_\alpha^\lambda(\mathscr{N}(P))|f_{\mathscr{N}(P)}-f_{\mathscr{N}(Q)}|^p.
\end{align*}

Using  the Poincar\'{e} inequality \eqref{poincare} and the fact that $\#\{Q: Q\asymp P\}$ are uniformly bounded, where ``$\#$'' means ``cardinality'', we get
\begin{align*}
\sum_{P\in\dyadic_{d, 2^k}} H^1_P&\lesssim \int_{\real^d} |{\tilde f}(x)-\mathscr T_{2^k}f(x)|^{ p}\, dm_d(x)+ (2^k+2)^{-\lambda}\int_{\bigcup_{2^{-2^k-1}\leq\ell(Q)\leq 2^{-2^k+1}}\mathscr N(Q)} |\nabla f|^p\, d\mu_\alpha^\lambda\\
&=: I_k+ (2^k+2)^{-\lambda} I'_k.
\end{align*}

Since the domains of integration in $I'_k$'s have bounded overlap, we obtain that
$$\sum_{k\geq 0} I'_k\lesssim \|f\|^{  p}_{W^{1,p}(\real^{d+1}_+, \mu_\alpha^\lambda)}.$$

Next, we are going to estimate the term $I_k$. Since it follows from the triangle inequality and the Minkowski inequality that
\begin{align*}
(I_k)^{1/p}&\leq \left(\int_{\real^d}\left(\sum_{n\geq 2^k}|\mathscr{T}_{n+1}f-\mathscr{T}_{n}f| \right)^p\, dm_d\right)^{1/p}\leq \sum_{n\geq 2^{k}}\left(\int_{\mathbb{R}^{d}}|\mathscr{T}_{n+1}f-\mathscr{T}_{n}f|^pdm_{d}\right)^{1/p},
\end{align*}
we know from the estimate \eqref{estimate-T_k-T} that
\[\int_{\real^d}  |\Trace_{n+1} f-\Trace_n f|^p\, dm_d\lesssim  (2+n)^{-\lambda}\sum_{Q\in \dyadic_{d, n}\cup\dyadic_{d, n+1}} \int_{\mathscr N(Q)} |\nabla f|^p\, d\mual.\]

For the case when $p>1, \alpha=p-1$ and $0<\gamma<\lambda-(p-1)$, it follows from the H\"{o}lder inequality that
\begin{align*}
I_k&\lesssim\left(\sum_{n\geq 2^k} (n+2)^{-(\lambda-\gamma)/p}\left((n+2)^{-\gamma}\int_{\cup_{2^{-n-2}\leq\ell(Q)\leq 2^{-n+1}}\mathscr N(Q)}|\nabla f|^p\, d\mu_{\alpha}^{\lambda}\right)^{1/p} \right)^p\\
& \lesssim \sum_{n\geq 2^k}(n+2)^{-\gamma}\int_{\cup_{2^{-n-2}\leq\ell(Q)\leq 2^{-n+1}}\mathscr N(Q)}|\nabla f|^p\, d\mu_{\alpha}^{\lambda},
\end{align*}
since the fact $\lambda-\gamma>p-1$ implies that the series $\sum_{n=1}^{\infty}(n+2)^{-\frac{\lambda-\gamma}{p-1}}$ is convergent.

For the case when $p=1$ and $0<\gamma\leq\lambda$, it is obvious that
$$I_k\lesssim \sum_{n\geq 2^k}(n+2)^{-\gamma}\int_{\cup_{2^{-n-2}\leq\ell(Q)\leq 2^{-n+1}}\mathscr N(Q)}|\nabla f|\, d\mu_{0}^{\lambda}.$$
Hence the Fubini theorem and the above estimates guarantee that for $p\geq 1$,
\[\begin{split}
\sum_{k\geq0}(2^{k}+2)^{\lambda}I_{k}
\lesssim&\sum_{k\geq0}(2^{k}+2)^{\gamma}\sum_{n\geq2^{k}}(n+2)^{-\gamma}\int_{\cup_{2^{-n-2}\leq\ell(Q)\leq 2^{-n+1}}\mathscr N(Q)}|\nabla f|^p\, d\mu_{\alpha}^{\lambda}\\
=&\sum_{n\geq1}(n+2)^{-\gamma}\int_{\cup_{2^{-n-2}\leq\ell(Q)\leq 2^{-n+1}}\mathscr N(Q)}|\nabla f|^p\, d\mu_{\alpha}^{\lambda}\sum_{0\leq2^{k}\leq n}(2^{k}+2)^{\gamma}\\
\lesssim&\sum_{n\geq0}\int_{\cup_{2^{-n-2}\leq\ell(Q)\leq 2^{-n+1}}\mathscr N(Q)}|\nabla f|^p\, d\mu_{\alpha}^{\lambda}\\
\lesssim&\|f\|^{ p}_{W^{1, p}(\mathbb{R}^{d+1}_{+},\mu_{\alpha}^{\lambda})}.
\end{split}\]
Thus the estimate of $H^1$ is given by
\begin{align*}
H^1&= \sum_{k=0}^{\infty} (2^k+2)^\gamma\sum_{P\in \dyadic_{d, 2^k}} H^1_P\lesssim \sum_{k\geq 0} (2^k+2)^\gamma I_k+ \sum_{k\geq 0} I'_k\lesssim \|f\|^{  p}_{W^{1, p}(\mathbb{R}^{d+1}_{+},\mu_{\alpha}^{\lambda})}.
\end{align*}

For the estimate of $H^2$, again, it follows from the Fubini theorem that
\begin{align*}
H^2&=\sum_{k\geq 0}  (2^k+2)^\gamma\sum_{P\in \dyadic_{d, 2^k}} \sum_{\substack{Q\asymp P\\ Q\in \dyadic_{d, 2^{k-1}}}} \int_P |{\tilde f}-{\tilde f}_Q|^p\, dm_d\\
&= \sum_{k\geq -1} (2^{k+1}+2)^\gamma \sum_{Q\in \dyadic_{d, 2^k} } \sum_{\substack{P\asymp Q\\ P\in \dyadic_{d, 2^{k+1}}}}  \int_P |{\tilde f}-{\tilde f}_Q|^p\, dm_d\\
&\leq \sum_{k\geq -1} (2^{k+1}+2)^\gamma \sum_{Q\in \dyadic_{d, 2^k} } \sum_{\substack{Q'\asymp Q\\ Q'\in \dyadic_{d, 2^{k}}}}  \int_{Q'} |{\tilde f}-{\tilde f}_Q|^p\, dm_d,
\end{align*}
where in the last inequality, the fact that for any $Q\in \dyadic_{d, 2^k}$,
\[\bigcup_{\substack{P\asymp Q\\ P\in \dyadic_{d, 2^{k+1}}}} P\subset \bigcup_{\substack{Q'\asymp Q\\ Q'\in \dyadic_{d, 2^{k}}}} Q' \] is applied.

By the reflexivity of the relation $Q'\asymp Q$ when $Q', Q\in \dyadic_{d, 2^k}$, we obtain that
\begin{align*}
H^2&\leq 3^\gamma\sum_{Q\in \dyadic_{d, 2^{-1}} } \sum_{\substack{Q'\asymp Q\\ Q'\in \dyadic_{d, 2^{-1}}}}  \int_{Q'} |{\tilde f}-{\tilde f}_Q|^p\, dm_d +H^1\\
&\lesssim \sum_{Q\in \dyadic_{d, 1}} \int_Q |{\tilde f}|^p\, dm_d+H^1\leq \|\mathscr T f\|^{ p}_{L^p(\real^d)}+H^1\lesssim \|f\|^{ p}_{W^{1, p}(\mathbb{R}^{d+1}_{+},\mu_{\alpha}^{\lambda})}.
\end{align*}

By combining the estimates of $H^1$ and $H^2$, we obtain that
$$\|\mathscr Tf\|_{\dot{\mathcal B}^{ \gamma}_{p}(\real^d)}\lesssim\|f\|_{W^{1, p}(\mathbb{R}^{d+1}_{+},\mu_{\alpha}^{\lambda})}=\|f\|_{\Sobolevpp},$$
which proves the theorem.\qed

\section{Extension operators}\label{sec-4}
The purpose of this section is to prove Theorem \ref{thm-1.4}.

Let $f\in \mathcal B^{\lambda}_{p}(\mathbb{R}^{d})$ and $\alpha=p-1$. For convenience, we rewrite $\mathscr Ef$ (see Definition \ref{extension-defn}) as follows:
\begin{equation*}\label{extension-eq}
\mathscr{E}f(x)=\sum_{k\geq0}\sum_{j=2^{k}}^{2^{k+1}-1}\sum_{Q\in \dyadic_{d, j}}\left(\vint_{\tdQ} f\, dm_{d}\right)\psi_{Q}(x).
\end{equation*}
Here we recall that $\tdQ \in \dyadic_{d, 2^k}$ is the unique cube with $Q\subset\tdQ$ for any $Q\in \bigcup_{j=2^k}^{2^{k+1}-1}\dyadic_{d, j}$.

The first task is to estimate the $L^{p}(\mathbb{R}^{d+1}_{+},\mu_{\alpha}^{\lambda})$-norm of $\mathscr{E}f$. It follows directly that
\begin{align*}
\|\mathscr{E}f\|_{L^{p}(\mathbb{R}^{d+1}_{+},\mu_{\alpha}^{\lambda})}^{p}&=\int_{\mathbb{R}^{d+1}_{+}}|\mathscr{E}f|^{p}d\mu_{\alpha}^{\lambda}=\sum_{n\geq1}\sum_{P\in\dyadic_{d,n}}\int_{\supp\psi_{P}}|\mathscr{E}f|^{p}d\mu_{\alpha}^{\lambda}\\
&=\sum_{P\in \dyadic_{d, 1}} \int_{\supp\psi_{P}}|\mathscr{E}f|^{p}d\mu_{\alpha}^{\lambda}+\sum_{n\geq2}\sum_{P\in\dyadic_{d,n}}\int_{\supp\psi_{P}}|\mathscr{E}f|^{p}d\mu_{\alpha}^{\lambda}\\
&=:  I_1+I_2.
\end{align*}

To estimate $I_1$, notice that for any $P\in\dyadic_{d,1}$ and each $x\in \supp\psi_{P}$, we have $\psi_{Q}(x)\neq0$ only for $Q\in\dyadic_{d,j}$ with $j=1,2$ and $Q\sim P$. Since $\mual(\supp \psi_p)\approx 1$ and $|\psi_Q|\leq 1$ for any $P\in \dyadic_{d, 1}$ and $Q\in \dyadic^1_d$, we obtain
\begin{align*}
I_1=\sum_{P\in\dyadic_{d,1}}\int_{\supp\psi_{P}}|\mathscr{E}f|^{p}d\mu_{\alpha}^{\lambda}
=&\sum_{P\in\dyadic_{d,1}}\int_{\supp\psi_{P}}\Big|\sum_{\substack{Q\in\dyadic_{d, 1}\cup\dyadic_{d, 2}\\ Q\sim P}}f_{\tdQ}\,\psi_{Q}(x)\Big|^{p}d\mu_{\alpha}^{\lambda}\\
\lesssim&\sum_{Q\in \dyadic_{d, 1}\cup\dyadic_{d, 2}}\int_{Q}|f|^{p}dm_{d}\lesssim\|f\|_{L^{p}(\mathbb{R}^{d})}^{p},
\end{align*}
where in the second last inequality, the fact that $\tdQ=Q$ for any $Q\in \dyadic_{d, 1}\cup\dyadic_{d, 2}$ is used.

Towards the estimate of $I_2$, since
\[\bigcup_{n\geq 2}\dyadic_{d, n}=\bigcup_{k\geq1}\bigcup_{2^k\leq j\leq 2^{k+1}-1} \dyadic_{d, j},\]
we have
\begin{align*}
I_2&=\sum_{k\geq 1}\sum_{j=2^k}^{2^{k+1}-1}\sum_{P\in \dyadic_{d, j}} \int_{\supp \psi_P} |\mathscr E f|^p\, d\mual\\
&= \sum_{k\geq 1}\sum_{P\in \dyadic_{d, 2^k}}\int_{\supp \psi_P} |\mathscr E f|^p\, d\mual +\sum_{k\geq 1}\sum_{P\in \dyadic_{d, 2^{k+1}-1}}\int_{\supp \psi_P} |\mathscr E f|^p\, d\mual\\
&\quad\quad\quad \quad\quad+ \sum_{k\geq 2}\sum_{j=2^k+1}^{2^{k+1}-2}\sum_{P\in \dyadic_{d, j}} \int_{\supp \psi_P} |\mathscr E f|^p\, d\mual\\
&=:I_2^A+I_2^B+I_2^C.
\end{align*}

For the estimate of $I_2^A$, recall the relation \eqref{eq:bump-support} and notice that for $P\in \dyadic_{d, 2^k}$, if $Q\sim P$, then
$$Q\in\dyadic_{d, 2^k-1}\cup\dyadic_{d, 2^k}\cup\dyadic_{d, 2^k+1}.$$
By the definition of $\tdQ$, we have
\begin{equation*}
\left\{\begin{array}{cl}
\tdQ\in \dyadic_{d, 2^k}\,\,&\mbox{if}\;\; Q\in \dyadic_{d, 2^k}\cup\dyadic_{d, 2^k+1},\\
\tdQ\in \dyadic_{d, 2^{k-1}}\,\,&\mbox{if}\;\; Q\in \dyadic_{d, 2^k-1}.
\end{array}\right.
\end{equation*}
Then it follows from the uniformly boundedness of $\#\{Q: Q\sim P\}$ that
\begin{align*}
I_2^A &\lesssim\sum_{k\geq 1}\sum_{P\in \dyadic_{d, 2^k}}\mual(\W(P))\sum_{Q\sim P}\vint_{\tdQ} |f|^{p}\, dm_{d}\\
&=\sum_{k\geq1}\sum_{P\in\dyadic_{d,2^{k}}}\mu_{\alpha}^{\lambda}(\mathscr{W}(P))\sum_{\substack{Q\sim P\\ Q\in \dyadic_{d, 2^k}\cup\dyadic_{d, 2^k+1}}}\vint_{\tdQ} |f|^{p}\, dm_{d}\\
&\quad\quad\quad\quad+\sum_{k\geq1}\sum_{P\in\dyadic_{d,2^{k}}}\mu_{\alpha}^{\lambda}(\mathscr{W}(P))\sum_{\substack{Q\sim P\\ Q\in \dyadic_{d, 2^k-1}}}\vint_{\tdQ} |f|^{p}\, dm_{d}\\
&\lesssim\sum_{k\geq1}2^{-2^{k}p}(2^{k}+2)^{\lambda}\sum_{Q^{\prime}\in\dyadic_{d,2^{k}}\cup\dyadic_{d,2^{k-1}}}\int_{Q^{\prime}}|f|^{p}dm_{d}\\
&\lesssim\sum_{k\geq1}2^{-2^{k}p}(2^{k}+2)^{\lambda}\int_{\mathbb{R}^{d}}|f|^{p}dm_{d}
\lesssim\|f\|_{L^{p}(\mathbb{R}^{d})}^{p},
\end{align*}
since $\mual(\W(P))\approx 2^{-2^k(d+p)} (2^{k}+2)^{\lambda}$ for $P\in \dyadic_{d, 2^k}$ and $\sum_{k\geq1}2^{-2^{k}p}(2^{k}+2)^{\lambda}$ is convergent.

The similar reasoning as in the estimate of $I_2^A$ ensures that
\[I_2^B\lesssim\sum_{k\geq 1} 2^{-2^{k+1}p}(2^{k+1}+2)^\lambda\sum_{Q'\in \dyadic_{d, 2^k}\cup\dyadic_{d, 2^{k+1}}} \int_{Q'}|f|^p\, dm_d  \lesssim \|f\|_{L^p(\real^d)}^p.\]

The estimate of $I_2^C$ is obtained by using the Fubini theorem as follows:
\begin{align*}
I_2^C&\lesssim \sum_{k\geq 2}\sum_{j=2^k+1}^{2^{k+1}-2}\sum_{P\in \dyadic_{d, j}}\mual(\W(P))\sum_{Q\sim P}\vint_{\tdQ} |f|^{p}\, dm_{d}\\
& = \sum_{k\geq 2} \sum_{Q'\in \dyadic_{d, 2^k}} 2^{ 2^kd}\int_{Q'} |f|^p\, dm_d\, \Bigg(\sum_{j=2^k+1}^{2^{k+1}-2}\sum_{\substack{P\in \dyadic_{d, j}, P\sim Q\\ \tdQ=Q'}} \mual(\W(P)) \Bigg)\\
&\approx \sum_{k\geq 2} \sum_{Q'\in \dyadic_{d, 2^k}} 2^{ 2^kd}\int_{Q'} |f|^p\, dm_d\, \Bigg(\sum_{j=2^k+1}^{2^{k+1}-2} 2^{-j(d+p)}(2+j)^\lambda\left(\frac{2^{-2^k}}{2^{-j}}\right)^d\Bigg)\\
&\lesssim\sum_{k\geq2}2^{-2^{k}p}(2^k+2)^\lambda\sum_{Q^{\prime}\in\dyadic_{d,2^{k}}}\int_{Q^{\prime}}|f|^{p}dm_{d}\lesssim\|f\|_{L^{p}(\mathbb{R}^{d})}^{p}.
\end{align*}

By combining the estimates of $I_1$, $I_2^A$, $I_2^B$ and $I_2^C$, we have
 \begin{equation}\label{extension-L-norm}
 \|\mathscr{E}f\|_{L^{p}(\mathbb{R}^{d+1}_{+},\mu_{\alpha}^{\lambda})}\lesssim\|f\|_{L^{p}(\mathbb{R}^{d})}.\end{equation}

The next task is to estimate the $L^{p}(\mathbb{R}^{d+1}_{+},\mual)$-norm of $\nabla(\mathscr{E}f)$. To reach this goal, we first divide $\mathbb{R}^{d+1}_{+}$ into two parts:
$$L^{p}(\mathbb{R}^{d+1}_{+},\mual)=X_{1}\cup X_{2},$$
where
$X_{1}=\bigcup_{k\geq1}\bigcup_{Q\in\dyadic_{d,k}}\mathscr{W}(P)$ and $X_{2}=\mathbb{R}^{d+1}_{+}\backslash X_{1}.$

If $x\in X_{2}$, we have $\psi_{Q}(x)\neq0$ only for $Q\in\dyadic_{d,1}$. So it follows from the Lipschitz continuity of the functions $\psi_Q$ that
$$
|\nabla(\mathscr{E}f)(x)|\leq|\Lip(\mathscr{E}f)(x)|\lesssim\sum_{Q\in\dyadic_{d,1}}|f_{Q}|\chi_{\supp\psi_{Q}}(x).
$$
Since $\mual(\supp\psi_{Q})\approx1$ for each $Q\in\dyadic_{d,1}$, the above estimate yields that
\begin{equation}\label{eq-add-}
\int_{X_{2}}|\nabla(\mathscr{E}f)|^{p}d\mu_{\alpha}^{\lambda}\lesssim\sum_{Q\in\dyadic_{d,1}}|f_{Q}|^{p}\mu_{\alpha}^{\lambda}(\supp\psi_{Q})\lesssim\|f\|_{L^{p}(\mathbb{R}^{d})}^{p}.
\end{equation}

If on the other hand $x\in X_{1}$, then there exist a $k\geq0$ and an unique cube $P\in\dyadic_{d,j}$ such that  $2^{k}\leq j<2^{k+1}$ and $x\in\mathscr{W}(P)$. So we can have
\begin{align*}
\int_{X_{1}}|\nabla(\mathscr{E}f)|^{p}d\mu_{\alpha}^{\lambda}&=\sum_{k\geq0}\sum_{j=2^{k}}^{2^{k+1}-1}\sum_{P\in\dyadic_{d,j}}\int_{\mathscr{W}(P)}|\nabla(\mathscr{E}f)|^{p}d\mu_{\alpha}^{\lambda}\\
& =\sum_{k\geq0} \sum_{P\in\dyadic_{d,2^k}}\int_{\mathscr{W}(P)}|\nabla(\mathscr{E}f)|^{p}d\mu_{\alpha}^{\lambda} +\sum_{k\geq1} \sum_{P\in\dyadic_{d,2^{k+1}-1}}\int_{\mathscr{W}(P)}|\nabla(\mathscr{E}f)|^{p}d\mu_{\alpha}^{\lambda} \\
&\quad\quad\quad\quad+\sum_{k\geq2}\sum_{j=2^k+1}^{2^{k+1}-2} \sum_{P\in\dyadic_{d,2^{k+1}-1}}\int_{\mathscr{W}(P)}|\nabla(\mathscr{E}f)|^{p}d\mu_{\alpha}^{\lambda}\\
& =: H^A+H^B+H^C.
\end{align*}

For the estimate of $H^A$, let $P\in\dyadic_{d,2^k}$ with $k\geq 0$. Then for each $x\in\mathscr{W}(P)$,
\begin{equation}\label{ext}
\mathscr{E}f(x)=\sum_{Q\sim P}\left(\vint_{\tdQ} f\, dm_{d}\right)\psi_{Q}(x).
\end{equation}
Again, it follows from the Lipschitz continuity of the functions $\psi_Q$ that
$$
|\nabla (\mathscr Ef)(x)|\leq |\Lip(\mathscr{E}f(x)-f_{P})|\lesssim\sum_{{Q\sim P}} \frac{1}{2^{-2^k}}|f_{\tdQ}-f_{P}|.
$$
Since for any $P\in \dyadic_{d, 2^k}$ and  $Q\sim P$, the cube $\tdQ$ is a selected neighbor of $P$, i.e., $\tdQ\asymp P$, we obtain
\begin{align*}
H^A&\lesssim \sum_{k\geq 0} \sum_{P\in\dyadic_{d, 2^k}} \mual(\W(P)) 2^{2^k p}\sum_{Q\sim P}|f_{\tdQ}-f_P|^p\\
&\approx \sum_{k\geq 0} \sum_{P\in\dyadic_{d, 2^k}} (2^k+2)^\lambda  m_d(P)\sum_{Q' \asymp P}|f_{Q'}-f_P|^p \\
&\leq \|f\|^{ p}_{\mathcal B^{\lambda}_{p}(\mathbb{R}^{d})}.
\end{align*}

For the estimate of $H^B$, let $P\in \dyadic_{d, 2^{k+1}-1}$. Then ${\mathcal S (P)} \in \dyadic_{d, 2^k}$ is the unique cube  with $P\subset {\mathcal S (P)}$. It follows from \eqref{ext} and the continuity of the functions $\psi_Q$ that
\[|\nabla(\mathscr Ef) (x)| \leq |\Lip(\mathscr{E}f(x)-f_{\mathcal S(P)})|\lesssim\sum_{{Q\sim P}} \frac{1}{2^{-2^{k+1}}}|f_{\tdQ}-f_{{\mathcal S (P)}}|.\]
Since for any $P\in \dyadic_{d, 2^{k+1}-1}$ and $Q\sim P$, the cube ${\mathcal S (P)}$ is a selected neighbor of
$\tdQ$, i.e., ${\mathcal S (P)}\asymp \tdQ$, it follows from the Fubini theorem that
\begin{align*}
H^B&\lesssim \sum_{k\geq 1} \sum_{P\in \dyadic_{d, 2^{k+1}-1}}\mual(\W(P)) 2^{2^{k+1}p} \sum_{Q\sim P} |f_{\tdQ}-f_{{\mathcal S (P)}}|\\
&\lesssim \sum_{k\geq 1} \sum_{j=2^{k+1}-2}^{2^k}\sum_{Q\in \dyadic_{d, j}} m_d(Q) (2^{k+1}+2)^\lambda\sum_{P'\asymp \tdQ} |f_{\tdQ}-f_{P'}|^p\\
&=\sum_{k\geq 1}(2^{k+1}+2)^\lambda \sum_{Q'\in \dyadic_{d, 2^{k}}\cup\dyadic_{d, 2^{k+1}}} \sum_{P'\asymp Q'} |f_{P'}-f_{Q'}|^p\Bigg(\sum_{j=2^{k+1}-2}^{2^k}\sum_{\substack{Q\in \dyadic_{d, j}\\  \tdQ=Q'}} m_d(Q)\Bigg)\\
&\lesssim \sum_{k\geq 1}(2^{k+1}+2)^\lambda \sum_{Q'\in \dyadic_{d, 2^{k}}\cup\dyadic_{d, 2^{k+1}}} m_d(Q')\sum_{P'\asymp Q'} |f_{P'}-f_{Q'}|^p\lesssim \|f\|^{  p}_{\mathcal B^{\lambda}_{p}(\mathbb{R}^{d})}.
\end{align*}

The left is to estimate $H^C$. For any $2^k+1\leq j\leq 2^{k+1}-2$ with $k\geq 2$, let
$$\dyadic_{d,j}=Y_{1,j}\cup Y_{2,j},$$
where $$
Y_{1,j}:=\{P:P\in\dyadic_{d,j}, \overline{P} \bigcap\overline{\mathbb{R}^{d}\backslash \tdP}=\emptyset\}
$$
and
$$
Y_{2,j}:=\{P:P\in\dyadic_{d,j}, \overline{P} \bigcap\overline{\mathbb{R}^{d}\backslash \tdP}\neq\emptyset\}.
$$

Let $x\in\mathscr{W}(P)$ for $P\in\dyadic_{d,j}$. If $P\in Y_{1,j}$, then the definition of $\mathscr{E}f$ implies that $\mathscr{E}f(x)\equiv f_{\tdP}$, and hence, $$\int_{\mathscr{W}(P)}|\nabla(\mathscr{E}f)|^{p}d\mu_{\alpha}^{\lambda}=0.$$
For the case when $P\in Y_{2,j}$, notice that for $Q\sim P$, the cube $\tdQ$ is a selected neighbor of $\tdP$, i.e., $\tdQ\asymp \tdP$.  It follows from \eqref{ext} that
\[|\nabla(\mathscr Ef) (x)| \leq |\Lip(\mathscr{E}f(x)-f_{\mathcal S(P)})|\lesssim\sum_{{Q\sim P}} \frac{1}{2^{-j}}|f_{\tdQ}-f_{{\mathcal S (P)}}|\leq \frac{1}{2^{-j}}\sum_{Q\asymp \tdP} |f_Q-f_{\tdP}| .\]
Hence, for any $P\in Y_{2, j}$, we have the estimate
\begin{align*}
\int_{\mathscr{W}(P)}|\nabla(\mathscr{E}f)|^{p}d\mu_{\alpha}^{\lambda}
\lesssim&\mu_{\alpha}^{\lambda}(\mathscr{W}(P))\bigg(\frac{1}{2^{-j}}\sum_{Q\asymp \tdP} |f_Q-f_{\tdP}|\bigg)^{p}\\
\lesssim&2^{-jd}(j+2)^{\lambda}\sum_{Q\asymp \tdP} |f_Q-f_{\tdP}|^{p}.
\end{align*}
It follows from the Fubini theorem that
\begin{align*}
H^C&=\sum_{k\geq 2} \sum_{j=2^{k}+1}^{2^{k+1}-2}\sum_{P\in Y_{1,j}}\int_{\mathscr{W}(P)}|\nabla(\mathscr{E}f)|^{p}d\mu_{\alpha}^{\lambda}+\sum_{k\geq 2} \sum_{j=2^{k}+1}^{2^{k+1}-2}\sum_{P\in Y_{2,j}}\int_{\mathscr{W}(P)}|\nabla(\mathscr{E}f)|^{p}d\mu_{\alpha}^{\lambda}\\
&=\sum_{k\geq 2} \sum_{j=2^{k}+1}^{2^{k+1}-2}\sum_{P\in Y_{2,j}} 2^{-jd}(j+2)^{\lambda}\sum_{Q\asymp \tdP} |f_Q-f_{\tdP}|^{p}\\
&\approx\sum_{k\geq 2} (2^k+2)^\lambda \sum_{P'\in \dyadic_{d, 2^k}} m_d(P') \sum_{Q\asymp P'}|f_Q-f_{P'}|^p\Bigg(\sum_{j=2^{k}+1}^{2^{k+1}-2}\sum_{\substack{P\in Y_{2,j}\\ P'=\tdP}} 2^{-jd}/m_d(P')\Bigg).
\end{align*}
Note that for each $P'\in\dyadic_{d,2^{k}}$, there are $\left(2^{(j-2^k)d}-(2^{j-2^k}-2)^d\right)$ many of $P\in Y_{2,j}$ such that $\tdP=P'$. Hence
\begin{align*}
\sum_{j=2^{k}+1}^{2^{k+1}-2}\sum_{\substack{P\in Y_{2,j}\\ P'=\tdP}} 2^{-jd}/m_d(P')&=\sum_{j=2^{k}+1}^{2^{k+1}-2} 2^{(2^k-j)d}\left(2^{(j-2^k)d}-(2^{j-2^k}-2)^d\right)\\
&\leq \sum_{i=0}^{\infty}1-(1-2^{-i})^d\\
&\leq \sum_{i=0}^{\infty} d\cdot 2^{-i}<\infty.
\end{align*}
This yields that
\[H^C\lesssim \sum_{k\geq 2} (2^k+2)^\lambda \sum_{P'\in \dyadic_{d, 2^k}} m_d(P') \sum_{Q\asymp P'}|f_Q-f_{P'}|^p\leq \|f\|^{  p}_{\mathcal B^{\lambda}_{p}(\mathbb{R}^{d})}.\]

By combining the estimates of $H^A$, $H^B$ and $H^C$, we  get that
\[
\int_{X_{1}}|\nabla(\mathscr{E}f)|^{p}d\mu_{\alpha}^{\lambda}
=H^A+H^B+H^C
\lesssim\|f\|_{\mathcal B^{\lambda}_{p}(\mathbb{R}^{d})}^{p}.
\]

By recalling the estimates \eqref{extension-L-norm} and \eqref{eq-add-}, we finally obtain that
$$\|\mathscr{E}f\|_{W^{1, p}(\mathbb{R}^{d+1}_{+},\mu_{p-1}^{\lambda})}=\|\mathscr{E}f\|_{W^{1, p}(\mathbb{R}^{d+1}_{+},\mu_{\alpha}^{\lambda})}\lesssim\|f\|_{\mathcal B^{\lambda}_{p}(\mathbb{R}^{d})}.$$

Since all $m_d$-almost every points in $\real^d$ are the Lebesgue ones of a function $f\in\mathcal B^{\lambda}_{p}(\mathbb{R}^{d})$, it is evident from the definition of the trace operator $\mathscr T$ that $\mathscr T(\mathscr Ef)=f$ for $m_d$-almost every points in $\real^d$ , and hence, the theorem is proved.
\qed

%\section{An example}\label{sec-5}
%
%The aim of this section is to construct an example to show that
%the conditions on $\lambda$ in Theorem \ref{Thm-1.2} are sharp.
%%for any $p>1$, the trace of $W^{1,p}(\real^{d+1}_+, \mu^0_{p-1})$ may not exist.

\section*{Acknowledgments}
The first author (Manzi Huang) was partly supported by NNSF of China under the number 11822105. The second author (Xiantao Wang) was partly supported by NNSFs of China under the numbers  12071121 and 11720101003 and the project under the number 2018KZDXM034. The third author (Zhuang Wang) was supported by NNSF of China under the number 12101226.

\medskip

\noindent Manzi Huang,

\noindent
MOE-LCSM, School of Mathematics and Statistics, Hunan Normal University, Changsha, Hunan 410081, People's Republic of
China, and School of Mathematical Science, Qufu Normal University, Qufu, Shangdong 273165,
People's Republic of China

\noindent{\it E-mail address}:  \texttt{mzhuang@hunnu.edu.cn}
\bigskip

\noindent Xiantao Wang,

\noindent
MOE-LCSM, School of Mathematics and Statistics, Hunan Normal University, Changsha, Hunan 410081, People's Republic of
China.

\noindent{\it E-mail address}:  \texttt{xtwang@hunnu.edu.cn}\bigskip

\noindent Zhuang Wang,

\noindent
MOE-LCSM, School of Mathematics and Statistics, Hunan Normal University, Changsha, Hunan 410081, People's Republic of
China.

\noindent{\it E-mail address}:  \texttt{zhuang.z.wang@foxmail.com}\bigskip

\noindent Zhihao Xu,

\noindent
MOE-LCSM, School of Mathematics and Statistics, Hunan Normal University, Changsha, Hunan 410081, People's Republic of
China.

\noindent{\it E-mail address}:  \texttt{734669860@qq.com}

%\section{Preliminaries}

\end{document}